\pdfoutput=1
\documentclass[10pt]{article}
\usepackage{authblk}
\usepackage{footnote}
\usepackage{color}
\usepackage{helvet}         
\usepackage{courier}        
\usepackage{type1cm}        
\usepackage{framed}
\usepackage{tikz}
\usepackage{makeidx}         
\usepackage{graphicx}        
\usepackage{multicol}        
\usepackage[bottom]{footmisc}
\usepackage{amsmath}
\usepackage{amssymb}
\usepackage{bbold}
\usepackage{amsthm}
\usepackage{subcaption}
\usepackage{sidecap}
\usepackage{floatrow}
\usepackage{comment}
\usepackage[font=small]{caption}
\usepackage[top=2cm, bottom=2cm, left=2cm, right=2cm]{geometry}
\newtheorem{theorem}{Theorem}

\newtheorem{lemma}[theorem]{Lemma}

\newtheorem{remark}[theorem]{Remark}

\numberwithin{theorem}{section}
\numberwithin{figure}{section}
\numberwithin{equation}{section}
\DeclareMathOperator{\SLE}{SLE}
\begin{document}
\title{A note on the GFF with one free boundary condition}
\author{Yong Han\footnote{College of Mathematics and Statistics, Shenzhen University, Shenzhen 518060, Guandong, China. hanyong@szu.edu.cn}
  \hspace{1cm} Yuefei Wang\footnote{College of Mathematics and Statistics, Shenzhen University, Shenzhen 518060, Guandong, China and Institute of Mathematics, Academy of Mathematics and Systems Sciences, Chinese Academy of
Sciences, Beijing 100190, China. wangyf@math.ac.cn}
\hspace{1cm} Zipeng Wang \footnote{College of Mathematics and Statistics, Chongqing University, Chongqing,
401331, China. zipengwang@cqu.edu.cn}
}
\date{}
\maketitle

\newcommand{\HH}{\mathbb{H}}

\newcommand{\R}{\mathbb{R}}
\newcommand{\C}{\mathbb{C}}
\newcommand{\N}{\mathbb{N}}
\newcommand{\Z}{\mathbb{Z}}
\newcommand{\E}{\mathbb{E}}
\newcommand{\PP}{\mathbb{P}}

\newcommand{\one}{\mathbb{1}}
\newcommand{\bn}{\mathbf{n}}
\newcommand{\MR}{MR}
\newcommand{\cond}{\,|\,}

\newcommand{\prob}{\mathbb{P}}

\begin{abstract}
In this note, we shall prove an explicit formula on the probability of the level line of the Gaussian Free Field (GFF) with mixed boundary condition terminating at the free boundary, which generalizes the results of GFF with Dirichlet boundary condition.
\end{abstract}

\section{Background}
Let us first recall some basic knowledge on Loewner chains (\cite{LawlerConformallyInvariantProcesses, LawlerSchrammWernerExponent1,RohdeSchrammSLEBasicProperty}).
Given a continuous function $W$ from $\R_+$ to $\R$: $t\rightarrow W_t$, the Loewner's  ordinary differential equation (ODE) is defined by
\begin{align}
\partial_t g_t(z)=\frac{2}{g_t(z)-W_t},\quad g_0(z)=z\in\HH.
\end{align}
For $z\in\overline{\HH}\setminus\{0\}$, let $\tau(z)$ be the blowup time of the Loewner's ODE above.
For $t\geq 0$, let $K_t$ be the set of $z\in\HH$ with $\tau(z)\leq t$. Then the collection $\{K_t: t\in\R_+\}$ forms an increasing family of compact subsets of $\overline{\HH}$ such that  for each $t$, $g_t$ is the unique conformal map from $\HH\setminus K_t$ onto $\HH$ satisfying $g_t(z)=z+2t/z+O(1/|z|^2)$ as $z\rightarrow \infty$. The conformal maps $(g_t)_{t\geq 0}$ (or the compact hulls $(K_t)_{t\geq 0}$ ) are called the Loewner chain driven by $W$.  For $\kappa>0$ and $B_t$ being the one dimensional standard Brownian motion, the random Loewner chain driven by $W_t=\sqrt{\kappa}B_t$ is called the Schramm Loewner Evolution ($\mathrm{SLE}_\kappa$) process. The $\mathrm{SLE}_\kappa$ process is generated by a curve $\gamma$ from $0$ to $\infty$, i.e., for each $t\geq 0$, $\HH\setminus K_t$ is the unbounded component of $\HH\setminus\gamma[0, t]$.  There is an variant of $\mathrm{SLE}_\kappa$ which is known as the $\mathrm{SLE}_\kappa(\rho)$ process. For $\rho\geq 0$, and $y\leq x\in\R$,  consider the following Stochastic differential equation  system:
\begin{align}\label{eqn::kapparho}
dW_t=\sqrt{\kappa}dB_t+\frac{\rho}{W_t-V_t}dt,\quad dV_t=\frac{2}{V_t-W_t},\quad W_0=x,\quad V_0=y.
\end{align}
The random Loewner chain driven by $W$ defined in \eqref{eqn::kapparho} is called the $\mathrm{SLE}_\kappa(\rho)$ process from $x$ to $\infty$ with force point $y$. More properties of $\mathrm{SLE}_\kappa(\rho)$ process can be found in \cite{MillerSheffieldIG1}.

Next, let us give a brief introduction to the GFF (\cite{DuplantierSheffieldLQGKPZ,MillerSheffieldIG1}).
For $a\in \R$, the Green function with Dirichlet boundary condition on $(a, \infty)$ and Neumann boundary  condition (also called free boundary condition ) on $(-\infty, a)$ is given by
\begin{align}
G_{\mathrm{mix}}(z, w):=\frac{1}{2\pi}\Re\left(\frac{\left(\sqrt{z-a}+\sqrt{w-a}\right)\left(\sqrt{z-a}-\overline{\sqrt{w-a}}\right)}{\left(\sqrt{z-a}-\sqrt{w-a}\right)\left(\sqrt{z-a}+\overline{\sqrt{w-a}}\right)}\right),\quad z, w\in\HH,
\end{align}
where we take the branch of the square root such that it takes value in the upper half plane.

The GFF with Dirichlet boundary condition on $(a, \infty)$ and Neumann boundary  condition (also called free boundary condition ) on $(-\infty, a)$ is a centered Gaussian process  $\Gamma$ indexed by the set of continuous functions with compact support in $\HH$ such that
\[
\E\left[\Gamma(f)\Gamma(g)\right]=\int_{\HH\times\HH}f(z)G_{\mathrm{mix}}(z, w)g(w)dzdw.
\]

Notice that when $a\rightarrow-\infty$, $\Gamma$ is reduced to the GFF with Dirichlet boundary condition.
In \cite{Peltola2017global} the authors gave a full description of the  crossing probabilities of all possible patterns of the level lines of the  GFF with alternating boundary conditions. In particular, they gave a formula of the probability that the level line of the GFF originated  from one given point terminates at another  given point ((1.8) of \cite[ Theorem 1.4]{Peltola2017global} ). In this note we will use martingale methods to  deal with the GFF with the mixture of one free boundary data and Dirichlet boundary data and give a formula corresponding to (1.8) of  \cite{Peltola2017global}.
\section{Results}
Given $n+1$ points $a<b_1<b_2<...<b_n$. Let $\phi$ be the harmonic function on $\HH$ satisfying the Neumann boundary condition on $(-\infty, a)$ and taking values $\pm \lambda$ alternatively in the intervals $(b_i, b_{i+1})$ for $i=0,1,2,...,n$, where we take $b_0=a, b_{n+1}=+\infty$ and $\lambda=\sqrt{\pi/8}$ (see the following figure).
\begin{center}
\includegraphics[scale=0.6]{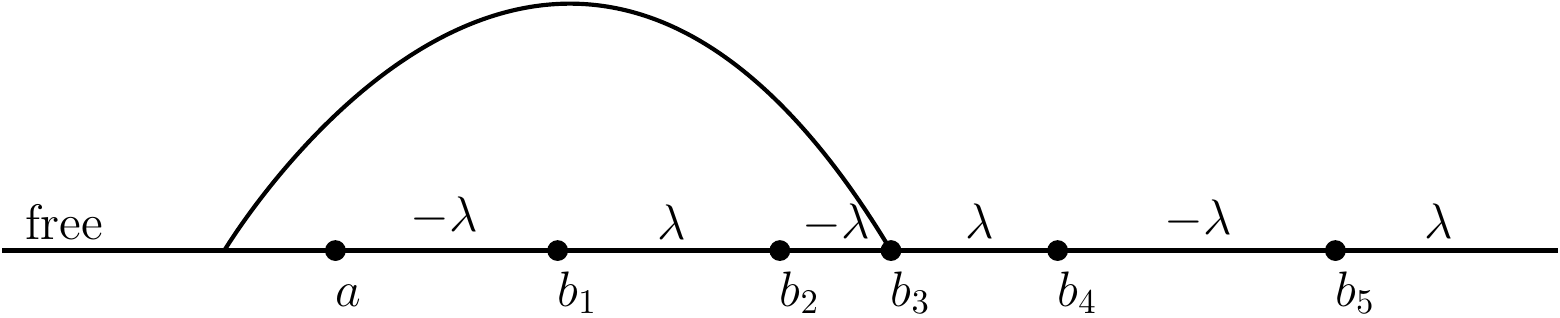}
\end{center}
The GFF (denoted by $h$) with free boundary condition on $(\infty, a)$ and alternating boundary conditions $\pm\lambda$ on $(b_i, b_{i+1})$ for $i=1,2,..., n$ is defined as the Gaussian process $\Gamma+\phi$, where $\Gamma$ is defined as the previous section.

Fix $k\in\{1,2,3,..., n\}$, let $I_k$ be the subset of $\{1, 2, 3,..., n\}$ that consists of numbers having the same parity as $k$ and $J_k=\{1,2,3,...,n\}\setminus I_k$. In \cite[Proposition 6]{IzyurovKytolaHadamardSLEFreeField}, the authors  gave a  coupling between the GFF in the strip with arbitrary mixed boundary condition which can be stated as following:
\begin{theorem}
There exists a coupling between $h$ and the random Loewner chain 
started from $b_k$
driven by the following  SDE systems:
\begin{align}\label{eqn::general_driven}
\begin{cases}
db_k(t)=2dB_t+\frac{-1}{b_k(t)-a(t)}dt+\sum_{i\in I_k}F(b_k(t), b_i(t), a(t))dt-\sum_{i\in J_k} F(b_k(t), b_j(t), a(t))dt\\[4mm]
da(t)=\frac{2}{a(t)-b_k(t)}dt,\quad db_i(t)=\frac{2}{b_i(t)-b_k(t)}dt,\quad \forall i\neq k,
\end{cases}
\end{align}
where $F$ is defined as
\begin{align}\label{eqn::special_Function}
F(x, y, z):=\frac{2}{x-y}\sqrt{\frac{y-z}{x-z}},\quad z<x, z<y.
\end{align}
\end{theorem}
Now we state our main result.
\begin{theorem}\label{thm::main}
The Loewner chain driven by $b_k(t)$ in \eqref{eqn::general_driven} is almost surely generated by a simple curve $\gamma[0,T]\rightarrow\overline{\HH}$ such that
\begin{align}
\gamma(0)=b_k,\quad \gamma(0, T)\subset\HH,\quad\gamma(T)=b_i \text{ for some } i\in J_k \text{ or }\gamma(T)\in (-\infty, a).
\end{align}
Moreover, we have
\begin{align}\label{eqn::martingale}
\mathbb{P}(\gamma(T)\in (-\infty, a))=\prod_{i\in I_k, i<k}\frac{1+\sqrt{\frac{b_i-a}{b_k-a}}}{1-\sqrt{\frac{b_i-a}{b_k-a}}}\prod_{i\in I_k, i>k}\frac{1+\sqrt{\frac{b_k-a}{b_i-a}}}{1-\sqrt{\frac{b_k-a}{b_i-a}}}\prod_{i\in J_k, i<k}\frac{1-\sqrt{\frac{b_i-a}{b_k-a}}}{1+\sqrt{\frac{b_i-a}{b_k-a}}}\prod_{i\in J_k, i>k}\frac{1-\sqrt{\frac{b_k-a}{b_i-a}}}{1+\sqrt{\frac{b_k-a}{b_i-a}}}.
\end{align}
\end{theorem}

We begin the proof with the following lemma.
\begin{lemma}
Suppose that $(g_t)_{t\geq 0}$ is $\mathrm{SLE}_4(-1)$ process from  $b_k$ to $\infty$ with force point $a$.  Define $b_j(t):=g_t(b_j)$ for $j\neq k$ and $a(t)=g_t(a)$.
Let $F$ be the function defined as \eqref{eqn::special_Function},
\begin{align}\label{eqn::Jtt}
J_t:=\sum_{i\in I_k}F(b_k(t), b_i(t), a(t))dt-\sum_{i\in J_k} F(b_k(t), b_j(t), a(t))dt,
\end{align}
and
\[
Z_t:=\left(\prod_{i\in I_k}\frac{\sqrt{b_i(t)-a(t)}-\sqrt{b_k(t)-a(t)}}{\sqrt{b_i(t)-a(t)}+\sqrt{b_k(t)-a(t)}}\prod_{j\in J_k}\frac{\sqrt{b_j(t)-a(t)}+\sqrt{b_k(t)-a(t)}}{\sqrt{b_j(t)-a(t)}-\sqrt{b_k(t)-a(t)}}\right)^2.
\]
Then $\log (Z_t)$ is a local martingale and $d\log (Z_t)=2J_t dB_t$.
\end{lemma}
\begin{proof}
By \cite[Theorem 1.3]{MillerSheffieldIG1} $\mathrm{SLE}_4(-1)$ is generated by a simple curve. So ${b_j(t):j\neq k}$ and $a(t)$ are well-defined.
By definition we have
\begin{align*}
\frac{1}{2}\log(Z_t)=&\sum_{i\in I_k}\log\left(\sqrt{b_i(t)-a(t)}-\sqrt{b_k(t)-a(t)}\right)-\log\left(\sqrt{b_i(t)-a(t)}+\sqrt{b_k(t)-a(t)}\right)\\
+&\sum_{j\in J_k}\log\left(\sqrt{b_j(t)-a(t)}+\sqrt{b_k(t)-a(t)}\right)-\log\left(\sqrt{b_j(t)-a(t)}-\sqrt{b_k(t)-a(t)}\right).
\end{align*}
So  from It\^{o}'s formula, we have
\begin{align*}
d\log (Z_t)=&\sum_{i\in I_k}\frac{-2}{(b_i(t)-b_k(t))(b_k(t)-a(t))}\sqrt{\frac{b_i(t)-a(t)}{b_k(t)-a(t)}}dt+\frac{2}{b_k(t)-b_i(t)}\sqrt{\frac{b_i(t)-a(t)}{b_k(t)-a(t)}}\left(2dB_t+\frac{-1}{b_k(t)-a(t)}\right)\\
+&\sum_{i\in J_k}\frac{2}{(b_i(t)-b_k(t))(b_k(t)-a(t))}\sqrt{\frac{b_i(t)-a(t)}{b_k(t)-a(t)}}dt+\frac{-2}{b_k(t)-b_i(t)}\sqrt{\frac{b_i(t)-a(t)}{b_k(t)-a(t)}}\left(2dB_t+\frac{-1}{b_k(t)-a(t)}\right)\\
=&\sum_{i\in I_k}\frac{2}{b_k(t)-b_i(t)}\sqrt{\frac{b_i(t)-a(t)}{b_k(t)-a(t)}}2dB_t+\sum_{i\in J_k}\frac{-2}{b_k(t)-b_i(t)}\sqrt{\frac{b_i(t)-a(t)}{b_k(t)-a(t)}}2dB_t\\
=&2J_tdB_t.
\end{align*}
\end{proof}
Now define $N_t:=\exp\{-\frac{1}{8}\int_0^t J_s^2ds\}$, where  $J_t$ is defined as  \eqref{eqn::Jtt}.
And it can be checked that $M_t:=Z_t^{\frac{1}{4}}N_t$ is a local martingale up to the first time $T$ that one of $\{a, b_1,...,b_{k-1}, b_{k+1},...,b_n\}$ is swallowed and $dM_t=\frac{1}{2}M_tJ_t dB_t$. In fact,
\begin{align}
T:=\inf\left\{t: b_k(t)-a(t)=0\text{ or } b_k(t)-g_t(b_i)=0 \text{ for some } i\neq k\right\}.
\end{align}
Now define
\begin{align}
T_n^{1}:=\inf\left\{t: \left|\frac{b_k(t)-a(t)}{b_i(t)-a(t)}\right|\leq \frac{1}{n}\text { or }\left|\frac{b_k(t)-a(t)}{b_i(t)-a(t)}\right|\geq 1-\frac{1}{n} \text{ for some }i>k \right\}.
\end{align}
\begin{align}
T_n^{2}:=\inf\left\{t: \left|\frac{b_i(t)-a(t)}{b_k(t)-a(t)}\right|\leq \frac{1}{n}\text { or }\left|\frac{b_i(t)-a(t)}{b_k(t)-a(t)}\right|\geq 1-\frac{1}{n} \text{ for some }i<k \right\}.
\end{align}
And $T_n:=T_n^{1}\wedge T_n^{2}$. Then we have $(M_t: 0\leq t\leq T_n)$ is a bounded martingale for any $n$. Since the Loewner chain driven by \eqref{eqn::general_driven} can be obtained by weighting $\SLE_4(-1)$ with the local martingale $M_t$. We can see that  the Loewner chain driven by \eqref{eqn::general_driven} is absolutely continuous with respect to $\SLE_4(-1)$ up to $T_n$ and therefore  the Loewner chain driven by \eqref{eqn::general_driven} is generated by a continuous curve up to $T_n$. Since $n$ is arbitrary, we have the Loewner chain driven by \eqref{eqn::general_driven} is generated by a continuous curve up to $T$, which we will denote by $\gamma$. We need to show that as $t\rightarrow T$,   $\gamma(t)$ converges almost surely.  Notice that when $\gamma(t)$ accumulates at $\{b_i: i\in I_k\}$, the local martingale is uniformly bounded and therefore it is absolutely continuous with respect to $\SLE_4(-1)$ up to and include $T$. But $\SLE_4(-1)$ is a continuous curve from $b_k$ to $(-\infty, a)$. So $\gamma$ can not accumulate at $\{b_i: i\in I_k\}$. The same reason excludes the event that $\gamma$ accumulates at $b_k$. So almost surely as $t\rightarrow T$, $\gamma(t)$ accumulates at $(\infty, a)\cup \{b_i: i\in J_k \}$.

Let $\mathcal{E}$ be the event that $\gamma$ has accumulation point in $(-\infty, a)$. Then on the event $\mathcal{E}$, $M_t$ is bounded and $\gamma$ is absolutely continuous with respect to $\SLE_4(-1)$ and therefore is a continuous curve from $b_k$ to $(-\infty, a)$.
We have finished the first part of Theorem \ref{thm::main}.

Now we prove \eqref{eqn::martingale}. Define
\begin{align}
g(a, b_1, b_2,..., b_n)=\prod_{i\in I_k, i<k}\frac{1+\sqrt{\frac{b_i-a}{b_k-a}}}{1-\sqrt{\frac{b_i-a}{b_k-a}}}\prod_{i\in I_k, i>k}\frac{1+\sqrt{\frac{b_k-a}{b_i-a}}}{1-\sqrt{\frac{b_k-a}{b_i-a}}}\prod_{i\in J_k, i<k}\frac{1-\sqrt{\frac{b_i-a}{b_k-a}}}{1+\sqrt{\frac{b_i-a}{b_k-a}}}\prod_{i\in J_k, i>k}\frac{1-\sqrt{\frac{b_k-a}{b_i-a}}}{1+\sqrt{\frac{b_k-a}{b_i-a}}}.
\end{align}
One can check that $\tilde{M}_t:=g(a(t), b_1(t), b_2(t),...,b_n(t))$ is a local martingale up to $T$.
\begin{lemma}
Take the notations as above,  $(\tilde{M}_t: 0\leq t\leq T)$ is a martingale and we have $\lim_{t\rightarrow T}\tilde{M}_t=1$ if $\gamma(T)\in (-\infty, a)$; $\lim_{t\rightarrow T}\tilde{M}_t=0$ otherwise.
\end{lemma}
\begin{proof}
Suppose that  $\gamma(T)=b_{i_0}$ for some $i_0\in J_k$. Without lost of generality, we may assume $b_{i_0}<b_k$. Then as $t\rightarrow T$ we have $b_j(t)-b_k(t)\rightarrow 0$ for any $i_0\leq j\leq k$. Hence we have $\prod_{i\in J_k, i>k}\frac{1-\sqrt{\frac{b_k(t)-a(t)}{b_i(t)-a(t)}}}{1+\sqrt{\frac{b_k(t)-a(t)}{b_i(t)-a(t)}}}$ tends to some finite value, so are
$$
\prod_{i\in I_k, i>k}\frac{1+\sqrt{\frac{b_k(t)-a(t)}{b_i(t)-a(t)}}}{1-\sqrt{\frac{b_k(t)-a(t)}{b_i(t)-a(t)}}},\quad \prod_{i\in I_k, i<i_0}\frac{1+\sqrt{\frac{b_i(t)-a(t)}{b_k(t)-a(t)}}}{1-\sqrt{\frac{b_i(t)-a(t)}{b_k(t)-a(t)}}},\quad \text{ and } \prod_{i\in J_k, i<i_0}\frac{1-\sqrt{\frac{b_i(t)-a(t)}{b_k(t)-a(t)}}}{1+\sqrt{\frac{b_i(t)-a(t)}{b_k(t)-a(t)}}}.
$$
We are left to investigate
\[
I:=\prod_{i\in I_k, i_0\leq i<k}\frac{1+\sqrt{\frac{b_i(t)-a(t)}{b_k(t)-a(t)}}}{1-\sqrt{\frac{b_i(t)-a(t)}{b_k(t)-a(t)}}}\prod_{i\in J_k, i_0\leq i<k}\frac{1-\sqrt{\frac{b_i(t)-a(t)}{b_k(t)-a(t)}}}{1+\sqrt{\frac{b_i(t)-a(t)}{b_k(t)-a(t)}}}.
\]
We can put the products in pairs $(i_0, i_0+1), (i_0+2, i_0+3),..., (k-3, k-2), (k-1)$. For each pair $(j, j+1)$, we have
\[
\frac{1-\sqrt{\frac{b_j(t)-a(t)}{b_k(t)-a(t)}}}{1+\sqrt{\frac{b_j(t)-a(t)}{b_k(t)-a(t)}}}\cdot\frac{1+\sqrt{\frac{b_{j+1}(t)-a(t)}{b_k(t)-a(t)}}}{1-\sqrt{\frac{b_{j+1}(t)-a(t)}{b_k(t)-a(t)}}}=\left(\frac{1+\sqrt{\frac{b_{j+1}(t)-a(t)}{b_k(t)-a(t)}}}{1+\sqrt{\frac{b_j(t)-a(t)}{b_k(t)-a(t)}}}\right)^2\frac{b_k(t)-b_j(t)}{b_k(t)-b_{j+1}(t)}
\]
Notice that  as $t\rightarrow T$
\[
\frac{b_k(t)-b_j(t)}{b_k(t)-b_{j+1}(t)}\rightarrow 1 \quad\forall i_0\leq j<k-1.
\]
So
\[
\lim_{t\rightarrow T}\prod_{i\in I_k, i_0\leq i<k}\frac{1+\sqrt{\frac{b_i(t)-a(t)}{b_k(t)-a(t)}}}{1-\sqrt{\frac{b_i(t)-a(t)}{b_k(t)-a(t)}}}\prod_{i\in J_k, i_0\leq i<k}\frac{1-\sqrt{\frac{b_i(t)-a(t)}{b_k(t)-a(t)}}}{1+\sqrt{\frac{b_i(t)-a(t)}{b_k(t)-a(t)}}}=\lim_{t\rightarrow T}\frac{1-\sqrt{\frac{b_{k-1}(t)-a(t)}{b_k(t)-a(t)}}}{1+\sqrt{\frac{b_{k-1}(t)-a(t)}{b_k(t)-a(t)}}}=0.
\]
We get  $\tilde{M_t}\rightarrow 0$ as $t\rightarrow T$ when $\gamma(T)=b_{i_0}$ for some $i_0\in J_k$.

Now suppose that we have $\gamma(T)\in (-\infty, a)$. Then as $t\rightarrow T$, $a(t)-b_k(t)\rightarrow 0$ and $b_i(t)-b_k(t)\rightarrow 0$ for $i<k$. So
\[
\prod_{i\in J_k, i>k}\frac{1-\sqrt{\frac{b_k(t)-a(t)}{b_i(t)-a(t)}}}{1+\sqrt{\frac{b_k(t)-a(t)}{b_i(t)-a(t)}}}\cdot \prod_{i\in I_k, i>k}\frac{1+\sqrt{\frac{b_k(t)-a(t)}{b_i(t)-a(t)}}}{1-\sqrt{\frac{b_k(t)-a(t)}{b_i(t)-a(t)}}}\rightarrow 1.
\]
By Lemma B.2 of  \cite{Peltola2017global}, we have  as $t\rightarrow T$,
\[
\frac{b_i(t)-a(t)}{b_k(t)-a(t)}\rightarrow 0,\quad \forall i<k.
\]
Therefore if $\gamma(T)\in (-\infty, a)$,  $\tilde{M}_t\rightarrow 1$ as $t\rightarrow T$. So $(\tilde{M}_t: 0\leq t\leq T)$ is a uniformly bounded martingale and by the stopping theorem we can get \eqref{eqn::martingale}.
\end{proof}
\begin{remark}
Since $k$ is arbitrary, we can get all the probabilities  that the level line started from $b_k$ terminates at the free boundary arc $(-\infty, a)$. Let $a$ goes to $\infty$, we can get the crossing probabilities of the GFF with Dirichlet boundary condition which is the Theorem 1.4 of \cite{Peltola2017global}. For example, take $k=1$ and $n=2N-1$, we have
\[
g(a, b_1, b_2,..., b_n)=\prod_{i\in I_k, i>k}\frac{1+\sqrt{\frac{b_k-a}{b_i-a}}}{1-\sqrt{\frac{b_k-a}{b_i-a}}}\prod_{i\in J_k, i>k}\frac{1-\sqrt{\frac{b_k-a}{b_i-a}}}{1+\sqrt{\frac{b_k-a}{b_i-a}}}.
\]
By taking $a\rightarrow\infty$,  the limit  above is equal to
\[
\frac{b_2-b_1}{b_3-b_1}\cdot\frac{b_4-b_1}{b_5-b_1}\cdot...\cdot\frac{b_{2N-2}-b_1}{b_{2N-1}-b_1},
\]
which is equation $(1.8)$ of \cite{Peltola2017global}.
\end{remark}
\section*{Acknowledgement}
 Y.Han and Y.Wang are supported in part by NSF of China (Grants No. 11688101) and Z.Wang is supported by NSF of China (Grants No.11601296)

\end{document}